\documentclass[11pt]{amsart}
\usepackage{amssymb,amsmath,amsthm}
\usepackage{a4wide}
\usepackage{graphicx}

\usepackage{color}
\usepackage{mathrsfs}
\usepackage{latexsym}

\numberwithin{equation}{section}

\theoremstyle{plain}
\begingroup
\newtheorem{theorem}{Theorem}[section]
\newtheorem{lemma}[theorem]{Lemma}
\newtheorem{proposition}[theorem]{Proposition}
\newtheorem{corollary}[theorem]{Corollary}
\endgroup

\theoremstyle{definition}
\begingroup

\newtheorem{remark}[theorem]{Remark}

\endgroup

\newcommand{\Gra}{{L}^{\neq\Delta}_{\ep,\delta}}
\newcommand{\Pic}{{S}^{\neq\Delta}_{\ep,\delta}}
\newcommand{\Ugra}{{{\Lambda}}_{\ep,\delta}}
\newcommand{\Upic}{{\Sigma}_{\ep,\delta}}

\newcommand{\giu}{\Omega_{\ep}^{\eta-}}
\newcommand{\sba}{\Omega_{\ep}^{\eta+}}

\newcommand{\tsba}{\tilde\Omega_{\ep}^{\eta+}}

\newcommand{\BorBu}{B_{\ep,\delta}^{\eta-}}
\newcommand{\BorCa}{B_{\ep,\delta}^{\eta+}}
\newcommand{\BorMi}{B_{\ep,\delta}^{\eta,\mathrm{n}}}
\newcommand{\BorGra}{\partial_{\ep,\delta}^{\eta+}}
\newcommand{\NoBorGra}{\not\partial_{\ep,\delta}^{\eta+}}

\newcommand{\F}{\mathcal{F}}

\newcommand{\A}{\mathcal A}
\newcommand{\M}{\mathcal M}

\newcommand{\N}{\mathbb{N}}
\newcommand{\Z}{\mathbb{Z}}
\newcommand{\R}{\mathbb{R}}

\newcommand{\C}{\mathcal{C}}

\newcommand{\E}{\mathcal{E}}
\newcommand{\av}{\mathrm{av}}

\newcommand{\supp}{\mathrm{supp}\;}
\newcommand{\ffi}{\varphi}
\newcommand{\ud}{\,\mathrm{d}}

\renewcommand\tilde{\widetilde}

\newcommand\e{\varepsilon}

\newcommand{\newatop}{\genfrac{}{}{0pt}{1}}
\newcommand{\res}{\mathop{\hbox{\vrule height 7pt width .5pt depth 0pt
\vrule height .5pt width 6pt depth 0pt}}\nolimits}

\def\A{\mathcal{A}}

\newcommand\ep{\varepsilon}

\newcommand{\weakly}{\rightharpoonup}
\newcommand{\weakstar}{\overset{\ast}{\rightharpoonup}}

\def\Om{\Omega}

\def\E{\mathcal{E}}

\def\H{\mathscr{H}}

\def\F{\mathcal{F}}
\def\O{\mathcal{O}}

\def\1{\mathbf{1}}
\def\loc{\mathrm{loc}}

\def\XXint#1#2#3{{\setbox0=\hbox{$#1{#2#3}{\int}$ }
\vcenter{\hbox{$#2#3$ }}\kern-.57\wd0}}

\def\eps{\varepsilon}

\DeclareMathOperator*{\argmin}{\arg\!\min}

\def\E{\mathcal{E}}

\newcommand{\I}{\mathcal{I}}
\newcommand{\Ed}{\mathrm{Ed}}
\newcommand{\Wi}{\mathrm{Wire}}
\newcommand{\ext}{\mathrm{ext}}
\newcommand{\inte}{\mathrm{int}}
\newcommand{\vth}{\vartheta}
\newcommand{\Per}{\mathrm{Per}}

\title[$\Gamma$-convergence of the Heitmann-Radin  energy to the crystalline perimeter] {$\Gamma$-convergence of the Heitmann-Radin sticky disc energy to the crystalline perimeter}

\author[L. De Luca]
{L. De Luca}
\address[Lucia De Luca]{
SISSA, Via Bonomea 265, I - 34136 Trieste, Italy
}
\email[L. De Luca]{ldeluca@sissa.it}

\author[M. Novaga]
{M. Novaga}
\address[Matteo Novaga]{
Dipartimento di Matematica, Universit\`a di Pisa, Largo Bruno Pontecorvo 5, I - 56127 Pisa, Italy
}
\email[M. Novaga]{matteo.novaga@unipi.it}

\author[M. Ponsiglione]
{M. Ponsiglione}
\address[Marcello Ponsiglione]{
Dipartimento di Matematica ``Guido Castelnuovo'', Sapienza Universit\`a di Roma, P. le Aldo Moro 5, I - 00185 Roma, Italy
}
\email[M. Ponsiglione]{ponsigli@mat.uniroma1.it}
\begin{document}

\begin{abstract}
We consider low energy   configurations  for the Heitmann-Radin sticky discs functional, in the limit of diverging number of discs.  More precisely, we renormalize the Heitmann-Radin potential by  subtracting the minimal energy per particle, i.e., the so called kissing number. 
For configurations whose energy scales like the perimeter, we prove a compactness result which shows the emergence of polycrystalline structures: 
The empirical measure converges to a set of finite perimeter, while a  microscopic variable, representing the orientation of the underlying lattice, 
converges to a locally constant function.   

Whenever the limit configuration is a single crystal, i.e., it has constant orientation, we show that the  $\Gamma$-limit is the anisotropic perimeter, corresponding to the Finsler metric determined by the orientation of the  single crystal.      
\end{abstract}

 \maketitle
 
\tableofcontents

\section*{Introduction}
Potentials that are attractive at long range and repulsive at very short range model many relevant systems and phenomena; among them, crystallization has a prominent place. 
A phenomenological potential with these features, particularly popular in Materials Science,  is the {\it Lennard-Jones} potential. 
Maybe the most basic potential mimicking attractive/repulsive interactions and leading to crystallization is the one proposed by  Heitmann and Radin \cite{HeRa}. 
In their model, particles are identified with sticky discs which maximize the number of their contact points without overlapping each other. 
More precisely, given $N$ discs in the plane, having diameter all equal to one and centered  in $x_1,\ldots,x_N$\,, the corresponding  Heitmann-Radin energy is given by
$$
E(x_1,\ldots,x_N) :=\frac 1 2\sum_{i\neq j}V(|x_j-x_i|)\,,
$$
where $V$ is the Heitmann-Radin  potential defined by
 $$
V(r):=\left\{\begin{array}{ll}
+\infty&\textrm{if }r<1\,,\\
-1&\textrm{if }r=1\,,\\
0&\textrm{if }r> 1\,.
\end{array}\right.
$$
In this paper we are interested in compactness and convergence results for almost minimizers of the energy $E$\,, in the limit as $N\to\infty$\,.
Before describing our approach we recall the main results about the minimizers of the energy $E$ for finite $N$ and on their behavior as $N\to\infty$\,.
In the seminal paper \cite{HeRa}, Heitmann and Radin prove that, for every fixed $N\in\N$\,, all the minimizers of the energy $E$ among the configurations $X=\{x_1,\ldots,x_N\}$ are subsets of an equilateral triangular lattice.
Their proof of this result relies on an ansatz on the exact value of the minimal energy which was previously  provided by Harborth \cite{Har}.
Moreover, the authors exhibit some explicit minimizers for all number $N$ of particles; such  minimizers  are regular hexagons with side $s$ whenever $N=N_s=1+6+\ldots+6s$\,, whereas, for general $N_s<N<N_{s+1}$\,, they are obtained by nestling the remaining discs around the  boundary of the regular hexagon constructed for $N_s$\,. 
Clearly, the empirical measures associated to such minimizers  converge (suitably scaled) to a macroscopic hexagon, referred to as {\it Wulff shape}.
However, the minimizing configurations are in general non-unique; in \cite{DLF2},  the authors characterize, through an explicit formula, all the number of particles $N$ for which the minimizer is  (up to a rotation and translation) unique.

In \cite{AFS} it is proven that, for any sequence of minimizers, the empirical measures converge to a Wulff shape. 
In \cite{S}, a refined  analysis   for minimizers of the energy $E$ for $N$ particles shows that the scaling law for the fluctuation about the asymptotic Wulff shape is $C\,N^{3/4}$ for some $C>0$\,, whereas in \cite{DPS} the optimal constant $C$ is explicitly provided.

It is well known that the Wulff shape is the solution of the isoperimetric problem for a suitable anisotropic perimeter. It is then clear the link between the Heitmann-Radin energy and perimeter-like functionals. This link has been exploited  in \cite{AFS} where  it is proven that, for configurations of $N$ particles lying on the triangular lattice and with prescribed energy upper bound scaling like a perimeter, the energy functionals  $\Gamma$-converge, as $N\to +\infty$\,,  to the anisotropic perimeter of the macroscopic shape.  Clearly, minimizing the $\Gamma$-limit with a volume constraint one obtains  the Wulff shape, and this gives back that the empirical measure of minimizers converge to the Wulff shape. In \cite{DLF1}, exploiting a discrete Gauss-Bonnet formula, for  finite $N$\,,   the energy of any configuration is rewritten in   terms of a suitable discrete notion of perimeter of the graph generated by the $N$ particles. 

In this paper we consider the asymptotic behavior of the Heitmann-Radin energy, in the perimeter-scaling regime, without assuming that the particles lie on a reference lattice. In this respect, we prove that the Heitmann-Radin energy enforces crystallization not only for minimizers, but also for low energy configurations. 
But while for minimizers the  orientation of the underlying lattice is constant, for almost minimizers  global  orientation  can be disrupted, giving rise to polycrystalline structures. 
Moreover, we  compute the $\Gamma$-limit of the energy functionals whenever the limiting orientation is constant, i.e., in the case of a single crystal. 

 We now describe in more details  our approach.    
Consider a configuration of $N$ particles. We recall that, for minimizers, the particles belong to a triangular lattice, and most of them (for large $N$) have exactly six nearest neighbors. In this respect, the minimal energy per particle is equal to $-6$\,, namely the opposite of the {\it kissing number}. Removing this bulk contribution  from the energy, a surface term  remains, which corresponds to the energy induced by the particles that have  less than six neighbors. 
At a first glance, these particles can be identified  as boundary particles. 

In order to introduce an internal variable, representing the local orientation of the crystal lattice, we observe that, at least for minimizers, 
most of the particles are vertices of some equilateral triangle. To these  triangles one can easily associate some orientation, for instance through the angles between its edges and  some reference straight line. 
Since triangular faces, edges and other 
geometrical objects 
play a role in our analysis, it is convenient to deal with the notion  of discrete graph generated by the particles; in this respect,  we will adopt the terminology and tools introduced in \cite{DLF1}. 

To any configuration of particles, we associate an empirical measure and a piecewise constant orientation, defined on the triangular faces of the graph. We prove that, in the perimeter-scaling energy regime, the empirical measures (suitably scaled) converge - up to a subsequence - to the characteristic function of some set $\Om$\,, representing the macroscopic (poly)crystal. Moreover, the regime we deal with provides uniform bounds for the $SBV$ norm of the function representing the microscopic orientation of the underlying lattice. In turn, we prove that the orientation converges to some 
limit function 
$\theta \in SBV(\Om)$\,, where $\theta=\sum_{j\in J}\theta_j\chi_{\omega_j}$ with $J   
\subseteq\N$ and 
 $\{\omega_j\}_{j}$ being a Caccioppoli partition of $\Omega$\,. 
 Here each  $\omega_j$ represents a grain of the polycrystal $\Om$\,, endowed with orientation $\theta_j$\,. 
 
 In the second part of the paper, we address the problem of computing the limit energy functional. We achieve this task in the case of a single crystal: If the orientation $\theta$ is constant, then the $\Gamma$-limit is given by the anisotropic perimeter of $\Om$\,, where the anisotropy corresponds to a Finsler metric whose Wulff shapes are hexagons with orientation  determined by $\theta$\,. This result clearly agrees with that of \cite{AFS}, the novelty being that here we do not assume that the particles belong to some reference lattice.
 The proof of the $\Gamma$-liminf inequality, without assuming crystallization  exploits the representation formulas, introduced in \cite{DLF1}, that allow to rewrite the Heitmann-Radin energy in terms of the discrete perimeter of the graph generated by the particles.

For polycrystals, where the orientation $\theta$ is not constant, one expects some additional surface contribution,  induced by grain boundaries. The sharp grain boundary energy, and in turn the $\Gamma$-limit in the general case, are not provided in this paper. Some upper and lower bounds are given in Proposition \ref{pcno}. Such bounds, although non optimal, are enough to show that, depending on the shape of the limit set $\Om$\,, both the single crystal and the polycrystal structure could be energetically favorable. 

A natural question is whether our results can be extended to more general interaction potentials, which are less rigid and take into account also elastic deformations. The  crystallization problem for  general potentials, both  for a finite and infinite number of particles,  is still an open research field which attracts much interest since decades \cite{BL}. 
For Lennard Jones type potentials, in \cite{T} it is proven  that the asymptotic energy density of minimizers is consistent with that of the regular triangular lattice. To our knowledge, our result is the first providing asymptotic (local) crystallization by compactness arguments for  almost minimizers of some explicit canonical, although very simple and rigid, interaction potential.   

The techniques and results developed in this paper share many analogies with  the so-called
{\it tessellation problems}. 
Among them we recall the  classical {\it honeycomb problem}, which consists in finding optimal clusters with minimal perimeter under volume constraints. Hexagonal tessellation is known to be optimal in the flat torus, thanks to the celebrated work of Hales \cite{H}. A more quantitative analysis of this result is developed in \cite{CarMag} and, in the framework of $\Gamma$-convergence, in \cite{AlbCar}.

In fact, our analysis also suggests new basic tessellation problems in $\Gamma$-convergence. A prototypical example is briefly described and analyzed in the Appendix, while further generalizations could deserve further investigations.

\section{Description of the problem}
In this section we introduce the notation we will use in the paper.
\subsection{The energy functionals.}
For every $\ep>0$ let $V_\ep:[0,+\infty)\to [0,+\infty]$ be the Heitmann-Radin sticky disc potential \cite{HeRa} defined by
$$
V_\ep(r):=\left\{\begin{array}{ll}
+\infty&\textrm{if }r<\ep\,,\\
-1&\textrm{if }r=\ep\,,\\
0&\textrm{if }r>\ep\,.
\end{array}\right.
$$

Given $X:=\{x_1,\ldots,x_N\}$ a finite subset of $\R^2$\,, the  Heitmann-Radin energy of $X$ is defined by
$$
E_\ep(X):=\frac 1 2\sum_{i\neq j}V_\ep(|x_j-x_i|)\,.
$$

Let $\M$ denote the class of Radon measures in $\R^2$ and let $\A$ be the class of empirical measures defined by 
$$
\A:=\Big\{\mu\in\M\,:\,\mu=\sum_{i=1}^N\delta_{x_i}\,,\,N\in\N\,,\,x_i\neq x_j\textrm{ for }i\neq j\Big\}\,.
$$

Note that there is a one-to-one correspondence $\I$ from $\A$ to the class of finite subsets of $\R^2$\,.
In view of this identification we can define the Heitmann-Radin  energies on measures by introducing the functionals $\E_\ep:\M\to [0,\infty]$ given by
\begin{equation}\label{defee}
\E_\ep(\mu)=\left\{\begin{array}{ll}
E_\ep(\I(\mu))&\textrm{if }\mu\in\A\,,\\
+\infty&\textrm{elsewhere.}
\end{array}\right.
\end{equation}
\subsection{Discrete graph representation.}

%
%
%
Let $\mu=\sum_{i=1}^{N}\delta_{x_i}\in\A$ be such that $\E_\ep(\mu)<+\infty$ and set $X=\I(\mu)$\,.
We say that $x_i$ and $x_j$ in $X$ are linked by an {\it edge}, or {\it bond}, if their mutual distance equals to $\ep$\, and we write $\{x_i,x_j\}$ for denoting such bond.  
We call $\Ed_\ep(X)$ the set of the bonds of $X$ and $(X,\Ed_\ep(X))$ the {\it bond graph} of the configuration.
 
\begin{figure}[h!]
{\def\svgwidth{400pt}
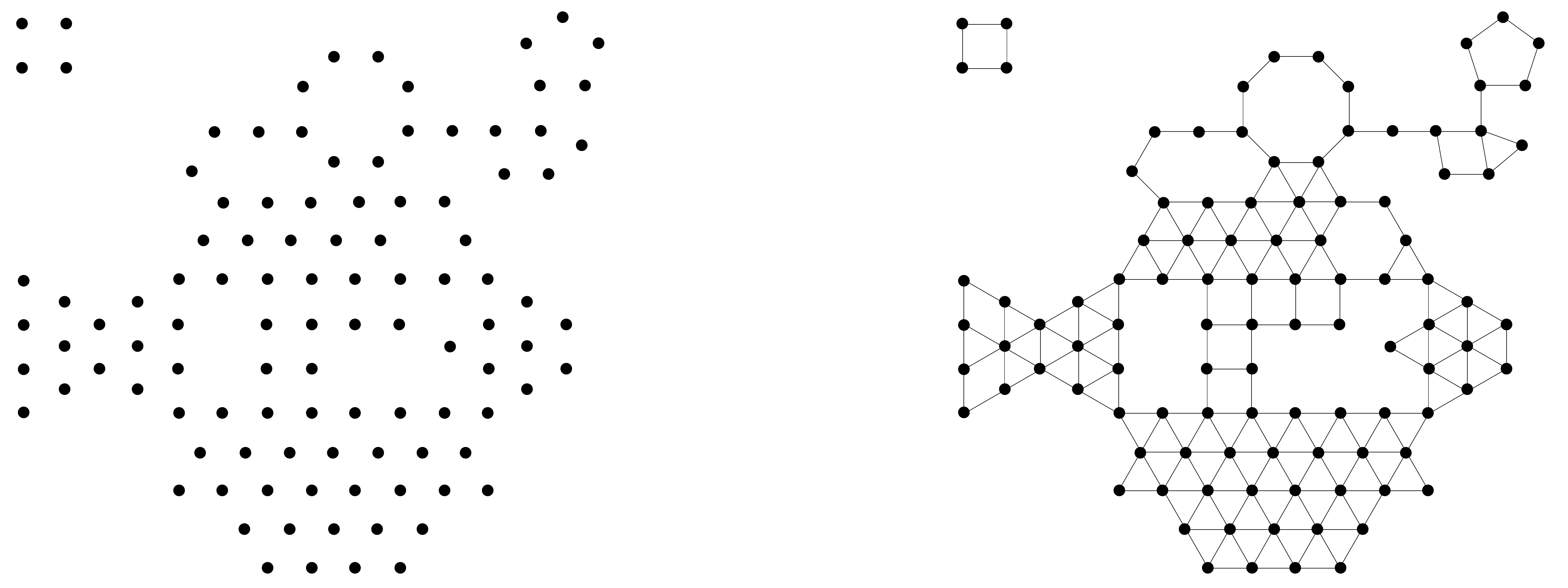}
\caption{Particle configuration and bond graph.}	\label{config_part}
\end{figure}

%
%
%
%
%
%

Since $\E_\ep(\mu)<+\infty$\,, simple geometric considerations easily imply that the bond graph is a {\it planar graph}, i.e., for any two different edges $\{x,y\}$ and $\{x',y'\}$\,, the corresponding line segments $[x,y]$ and $[x',y']$ do not cross.

It will be very useful to distinguish between ``interior'' and ``boundary'' edges. To this end we first provide the notion of {\it face} as it is introduced in \cite{DLF1}. 
By a face $f$ we mean any open and bounded subset of $\R^2$ which is nonempty, does not contain any point $x\in X$\,, and whose boundary is given by a cycle, i.e., $\partial f = \cup_{i=1}^k [x_{i-1},x_i]$ for some points $x_0,x_1,..,x_k=x_0\in X$ with $\{x_{i-1},x_i\}\in \Ed_\ep(X)$\,. Notice that the points $x_0,..,x_{k-1}$ do not need to be pairwise distinct, as a face might contain ``inner wire edges'' (see the definition below). Note also that for non-connected graphs, the definition above slightly differs from standard conventions because ring-shaped regions bounded by two cycles are not faces. 
We denote by $F_\ep(X)$ the set of faces of the bond graph $(X,\Ed_\ep(X))$\,.
Moreover, we define $F_\ep^\Delta(X)$ as the set of faces $f$ for which $k=3$ and $F_\ep^{\neq\Delta}(X):=F_\ep(X)\setminus F_\ep^\Delta(X)$\,.
\\[2mm]

Set $v_0(X):=\sharp X$\,, $v_1(X):=\sharp \Ed_\ep(X)$\,, and $v_2(X):=\sharp F_\ep(X)$\,, we define the Euler characteristic of the graph $(X,\Ed_\ep(X))$ as  
\begin{equation} \label{Euchar}
   \chi(X) := \sum_{k=0}^2 (-1)^k v_k(X)\,.
\end{equation}

Then we define the following sets:
\begin{itemize}
\item $\Wi(X)$ is the set of edges that either do not lie on the boundary of any face or
lie on the boundary of precisely one face but not on the boundary of its closure;

\item $\partial_{\ext}(X)$ is the set of edges lying on the boundary of precisely one face {and} on the boundary of its closure;

\item $\partial^1_{\inte}(X)$ is the set of edges lying on a triangular face and on a non triangular face;

\item $\partial^2_{\inte}(X)$ is the set of edges lying  on two non triangular faces.
\end{itemize}

By \cite[formula (3.7)]{DLF1}, we have
\begin{equation}\label{energy1}
E_\ep(X)+3 \sharp X 
= \sharp \partial_{\ext}(X)+ \sharp \partial^1_{\inte}(X)+2 \sharp \partial^2_{\inte}(X)-3\sharp F_\ep^{\neq\Delta}(X)+2\sharp \Wi(X)+3\chi(X) \,.
\end{equation}
Note that 
\begin{equation}\label{energy2}
E_\ep(X)+3 \sharp X =\frac 1 \ep \Per\Big(\bigcup_{f\in F_\ep(X)}f\Big)+\sum_{f\in F_\ep^{\neq \Delta}(X)}\Big(\frac{\Per(f)}{\ep}-3\Big) +2\sharp \Wi(X)+3\chi(X)\,,
\end{equation}
where $\Per(A)$ denotes the De Giorgi's perimeter of $A$ for every measurable set $A$\,.

With a little abuse of notation, we will often write $\Ed_\ep(\mu)$\,, $F_\ep(\mu)$\,, $F_\ep^\Delta(\mu)$ and $F_\ep^{\neq\Delta}(\mu)$ in place of $\Ed_{\ep}(\I(\mu))$\,, $F_\ep(\I(\mu))$\,, $F_\ep^\Delta(\I(\mu))$ and $F_\ep^{\neq\Delta}(\I(\mu))$ respectively.
\subsection{Grain orientations}
Let $\mu\in\A$ be such that $\E_\ep(\mu)<+\infty$\,. For every $\alpha\in\R$ we define
$$
P(\alpha):=\argmin\Big\{\Big|\alpha-k\frac{\pi}{3}\Big|\,:\,k\in\Z\Big\}\in\Z\,,
$$
with the convention that, if the $\argmin$ is not unique, then we choose the minimal one.
Clearly
\begin{equation}\label{orientabene}
P\Big(\alpha+j\frac{\pi}{3}\Big)=P(\alpha)+j\qquad\textrm{ for every }j\in\Z\,.
\end{equation}

Let $f\in F_\ep^{\Delta}(\mu)$ and let $w=e^{i\alpha_w}$ be a unit vector parallel to one of the edges of $f$ (with arbitrary orientation). We set
\begin{equation}\label{orient}
\alpha(f):=\alpha_w-P(\alpha_w)\frac{\pi}{3}\,\qquad \theta(f):=\alpha(f)+{\frac\pi 2}\,.
\end{equation}
Since all the edges of an equilateral triangle are obtained by rotating one fixed edge by an integer multiple of $\frac{\pi}{3}$\,, in view of \eqref{orientabene}, the definitions of $\alpha(f)$  and $\theta(f)$ in \eqref{orient} are well-posed. {Note also that $\theta(f)$ is the angle between $e_1$ and one of the medians of $f$\,.}
Moreover, by construction, $\alpha(f)\in (-\frac{\pi}{6},\frac{\pi}{6}]$ and hence $\theta(f)\in (\frac{\pi}{3},\frac{2}{3}\pi]$\,.

We set
\begin{equation}\label{deftheta}
\theta_\ep(\mu):=\sum_{f\in F_\ep^\Delta(\mu)}\theta(f)\chi_{f}\,.
\end{equation}

\subsection{Surface energy and Wulff shape}

Let us introduce a Finsler norm $\ffi$ whose unit ball is a unitary hexagon in $\R^2$\,. For every $\eta\in\R^2$ set
\begin{equation}\label{defffi}
\ffi(\eta):=\min\Big\{\sum_{j=1}^{3}|\lambda_j|\,:\,\eta=\sum_{j=1}^3\lambda_j v_j,\,\lambda_j\in\R\Big\}\,,
\end{equation}
where 
\begin{equation}\label{gener}
v_1=e^{i{\frac\pi 6}}\,,\quad v_{2}=e^{i\frac\pi 2}\,,\quad v_3=e^{i{\frac{5}{6}\pi}}\,.
\end{equation}
We define a one-parameter family of Finsler norms $\ffi_{\theta}$\,, for $\theta\in (\frac\pi 3, \frac 2 3\pi]$ by setting
\begin{equation}\label{defffian}
\ffi_\theta(\eta):=\min\Big\{\sum_{j=1}^{3}|\lambda_j|\,:\,\eta=\sum_{j=1}^3\lambda_j v_{j,\theta},\,\lambda_j\in\R\Big\}\,,
\end{equation}
where $v_{j,\theta}=e^{i(\theta-\frac \pi 2)}v_j$ for $j=1,2, 3$\,. Note that $\ffi_{\frac \pi 2}\equiv \ffi$\,.

For every set $G$ of finite perimeter, we set
$$
\Per_{\ffi_\theta}(G):=\int_{\partial^{*}G}\ffi_{\theta}(\nu)\ud \H^1\,,
$$
where $\nu$ denotes the outer normal to $\partial ^*G$ and $\H^1$ denotes the one dimensional Hausdorff measure.

We denote  by $W$  the regular hexagon centered at the origin with area equal to one, defined by
$$
W:= \Big\{ x\in \R^2: |x \cdot v_i| \le 2^{-\frac{1}{2}}  \, 3^{-\frac{1}{4}}, \, i=1,2,3\Big\},
$$

 and set $W_{\theta}:=  e^{i \theta} W$ for all $\theta \in\R$\,.  
These sets are referred to as {\it Wulff shapes}: it is well known \cite{FM} that they are the solutions of the isoperimetric inequality corresponding to the anisotropic perimeters $\ffi_\theta$\,.

%

\subsection{Preliminaries on $SBV$ functions}

We refer to the book \cite{AFP} for the definitions and the main properties of $BV$ and $SBV$ functions, sets of finite perimeter, and Caccioppoli partitions.
Here we list few preliminaries and properties  that will be useful in the following. We begin by recalling some standard notation.

Let $A\subseteq\R^2$ be open. As customary,  $BV(A)$ (resp. $SBV(A)$) denotes the set of functions of bounded variation (resp. special functions of bounded variation) defined on $A$ and taking values in $\R$\,. 
Moreover, $SBV_\loc(A)$ denotes the class of functions belonging to $SBV(A')$ for all open bounded sets $A'\subset\subset A$\,. Given any set $D\subset \R$\,, the classes of functions $BV(A;D)$\,,  $SBV(A;D)$ and $SBV_\loc(A;D)$ are defined in the obvious way. 

We say that a set $\Omega\subset\R^2$ has finite perimeter in $A$ if  $\chi_\Omega\in BV(A)$ and we denote by $\Per(\Omega,A)$ the relative perimeter of $\Omega$ in $A$\,. It is well known that $\Per(\Omega,A)= \H^1(\partial^*\Om\cap A)$\,, where 
$\partial^*$ denotes the reduced boundary. If $A=\R^2$ we simply say that $\Omega$ has finite perimeter and we denote by $\Per(\Omega)$ its perimeter.
Finally, if $\Omega$ is a set of finite perimeter,  a Caccioppoli partition of $\Omega$ is a countable partition $\{\omega_j\}_{j}$ of $\Omega$ into sets of (positive Lebesgue measure and) finite perimeter with $\sum_{j}\Per(\omega_j,\Omega)<\infty$\,. 

We recall that the distributional gradient $\mathrm{D} g$ of a function $g\in SBV(A)$ can be decomposed as:
\[
	\mathrm{D} g=\nabla g\,\mathcal L^2\res A+(g^+ - g^-)\otimes \nu_g \,\mathcal{H}^{1}\res{S_g}\,,
\]
where $\nabla g$ is the approximate gradient of $g$\,, $S_g$ is the jump set of $g$\,, $\nu_g$  is a unit normal to $S_g$ and $g^{\pm}$ are the approximate trace values of $g$ on  $S_g$\,.


We recall a compactness result.
\begin{theorem}[Compactness \cite{A}]\label{SBVComp}
	Let $A$ be bounded and let $\{g_h\}\subset SBV(A)$\,. Assume that there exists   $p > 1$ and $C>0$ such that
	\begin{equation}\label{compactnessCondition}
		\int_{A} |\nabla g_h|^p \ud x + \mathcal{H}^{1}(S_{g_h}) + \left\| g_h	\right\|_{L^\infty (A)}\leq C\quad\textrm{for all }h\in\N\,.
	\end{equation}
	Then, there exists $g \in SBV(A)$ such that, up to a subsequence,
\begin{equation}\label{sbvconv}
\begin{split}
	&g_h \to g \textrm{ (strongly) in } L^1(A)\,,\\
	&\nabla g_h \weakly \nabla g \textrm{ (weakly) in } L^1(A;\R^{2})\,,\\
	&\liminf_{h\to\infty} \mathcal H^1(S_{g_h}\cap A') \geq \mathcal H^1(S_g \cap A')\quad\textrm{for every open set }A'\subseteq A\,.
\end{split}
\end{equation}
\end{theorem}
In the following, we say that a sequence $\{g_h\} \subset SBV(A)$ weakly converges  in $SBV(A)$ to a function $g\in SBV(A)$\,, and we write that $g_h\weakly g$ in $SBV(A)$\,,  if $g_h$ satisfy \eqref{compactnessCondition} for some $p>1$ and $g_h\to g $ in $L^1(A)$\,.
The corollary below easily follows by Theorem \ref{SBVComp}.  
\begin{corollary}\label{SBVcompcor}
Let $\{g_h\}\subset SBV(\R^2)$\,. Assume that there exists   $p > 1$ and $C>0$ such that
	\begin{equation}\label{compactnessConditioncor}
		\int_{\R^2} |\nabla g_h|^p \ud x + \mathcal{H}^{1}(S_{g_h}) + \left\| g_h	\right\|_{L^\infty (\R^2)}\leq C\quad\textrm{for all }h\in\N\,.
	\end{equation}
Then, there exists $g \in SBV(\R^2)$ such that, up to a subsequence, \eqref{sbvconv} holds for every open bounded set $A\subset\R^2$\,.	
\end{corollary}
We say that $\{g_h\} \subset SBV_\loc(\R^2)$ weakly converges  in $SBV_\loc(\R^2)$ to a function $g\in SBV_\loc(\R^2)$\,, and we write that $g_h\weakly g$ in $SBV_\loc(\R^2)$\,, if  $g_h\weakly g$ in $SBV(A)$ for every open bounded set $A$\,.

\section{$\Gamma$-convergence analysis}
In this section we study the asymptotic beaviour, as $\e\to 0$\,, of the functionals  $\E_\ep$ defined in \eqref{defee}. More precisely, we consider the functionals  $\E_\ep(\mu)+{3}\,\mu(\R^2)$ and provide a compactness and a $\Gamma$-convergence result. 

\subsection{Compactness}
\begin{theorem}\label{compthm}
Let $\{\mu_\ep\}\subset\M$ be such that $\E_\ep(\mu_\ep)+{3}\,\mu_\ep(\R^2)\le \frac C \ep$\,. Then,  up to a subsequence, 
\begin{itemize}
\item[(i)] $\ep^2\frac{\sqrt 3}{2}\mu_\ep\weakstar\chi_\Omega \ud x$ for some set $\Omega\subset\R^2$ with $\chi_{\Omega}\in BV(\R^2)$\,;
\item[(ii)] $\theta_\ep(\mu_\ep)\weakly \theta$ in $SBV_\loc(\R^2)$\,, for some $\theta=\sum_{j\in J}\theta_j\chi_{\omega_j}$ in $SBV(\R^2)$\,, where $J\subseteq\N$\,,
 $\{\omega_j\}_{j}$ is a Caccioppoli partition of $\Omega$\,, and $\{\theta_j\}_{j}\subset (\frac\pi 3,\frac 2 3\pi]$\,.
\end{itemize}
\end{theorem}
\begin{proof}
The proof is divided into two steps.

{\it Step 1.}
We first prove that (ii) holds true for some set $\Om$ with finite perimeter. In view of the energy bound and \eqref{energy2} we have
\begin{align*}
C & \ge \ep ( \E_\ep(\mu_\ep)+{3}\,\mu_\ep(\R^2)) \ge 
\Per \Big(\bigcup_{f\in F_\ep(\mu_\ep)}f\Big)+ \frac{1}{4} \sum_{f\in F_\ep^{\neq \Delta}(\mu_\ep)}  \Per(f) 
\\
& \ge \frac{1}{4}  \Per \Big(\bigcup_{f\in F^\Delta_\ep(\mu_\ep)}f\Big) = \frac{1}{4} \Per(\Om_\ep),
\end{align*}
where we have set $\Om_\ep := \bigcup_{f\in F^\Delta_\ep(\mu_\ep)}f$\,. Then, the claim (ii) follows by the compactness statement (a) of Theorem \ref{teoapp}.

{\it Step 2.}
Now we prove (i) with $\Om$ provided in Step 1. To this purpose, for every $f\in F_\ep^{\Delta}(\mu_\ep)$ we denote by $a_j(f)$ ($j=1,2,3$) the vertices of $f$\,,  and we define
\begin{equation}
\hat\mu_\ep:=\frac 1 6\sum_{f\in F_{\ep}^{\Delta}(\mu_\ep)}\sum_{j=1}^3 \delta_{a_j(f)}\,,\qquad\tilde\mu_\ep:=\sum_{f\in F_{\ep}^{\Delta}(\mu_\ep)}\chi_f\,.
\end{equation}
By the energy bound and \eqref{energy2}, 
\begin{equation*}
C\ge\ep(\E_\ep(\mu_\ep)+{3}\,\mu_\ep(\R^2))\ge \Per\Big(\bigcup_{f\in F_\ep(\mu_\ep)}f\Big)
\end{equation*}
so that, by the isoperimetric inequality, we obtain
\begin{equation*}
\tilde\mu_\ep(\R^2)=|\Omega_\ep|\le \Big|\bigcup_{f\in F_\ep(\mu_\ep)}f\Big|\le C\,.
\end{equation*}
By the proof of Step 1 it follows that, up to a subsequence, $\tilde\mu_\ep\to \chi_{\Omega}$ in $L^1(\R^2)$\,. \\
We now show that $\ep^2 \frac{\sqrt{3}}{2} \hat \mu_\ep - \tilde \mu_\e \weakstar 0$\,. Let $\psi\in C^0_c(\R^2)$\,, and let $\psi_f$ be the average of $\psi$ on the triangle $f\in F_\ep^{\Delta}(\mu_\ep)$\,. Then,
\begin{equation}\label{cappmenotil}
\begin{aligned}
\left | \langle \ep^2 \frac{\sqrt{3}}{2} \hat \mu_\ep - \tilde \mu_\e, \psi\rangle \right |&= 
\left | \sum_{f\in F_{\ep}^{\Delta}(\mu_\ep)} \langle \ep^2 \frac{\sqrt{3}}{2} \hat \mu_\ep - \tilde \mu_\e, \psi \res f\rangle\right |
\\
&= \left |  \sum_{f\in F_{\ep}^{\Delta}(\mu_\ep)} \langle \ep^2 \frac{\sqrt{3}}{2} \hat \mu_\ep - \tilde \mu_\e, (\psi - \psi_f) \res f\rangle \right | \\
&\le 
 \sum_{f\in F_{\ep}^{\Delta}(\mu_\ep)} \left| \langle \ep^2 \frac{\sqrt{3}}{2} \hat \mu_\ep - \tilde \mu_\e, (\psi - \psi_f) \res f\rangle \right |\\
 & \le 2 |\tilde \mu_\e|(\R^2) \
 r_\e(\psi) \le C r_\e(\psi) \to 0\,,
\end{aligned}
\end{equation}
where $r_\e(\psi)$ is the modulus of continuity of $\psi$\,. \\
Now we prove that $\ep^2|\mu_\ep-\hat\mu_\ep|(\R^2)\to 0$\,.
We first notice that
$$
\sharp Y_\ep\le 2\ep(\E_\ep(\mu_\ep)+3\mu_\ep(\R^2))\,, 
$$
where $Y_\ep:=\{x\in\supp\mu_\ep\,:\,x\textrm { lies on at most five bonds}\}$\,.
As a consequence, by using the energy bound, we get
\begin{equation}\label{vermenotil}
\begin{aligned}
\ep^2|\mu_\ep-\hat\mu_\ep|(\R^2)=\ep^2 (\mu_\ep-\hat\mu_\ep)(\R^2)\le \ep^2\sharp(\supp\mu_\ep\setminus\supp\hat\mu_\ep)= \ep^2\sharp Y_\ep\le C\ep\to 0\,.
\end{aligned}
\end{equation}
By combining \eqref{cappmenotil} and \eqref{vermenotil} we obtain (i).
\end{proof}

\subsection{$\Gamma$-convergence}
%
%
\begin{theorem}\label{gammaconvHR}
The following $\Gamma$-convergence result holds true.
\begin{itemize}
\item [(i)] ($\Gamma$-liminf inequality) Let $\{\mu_\ep\}\subset\M$ satisfy (i) and (ii) of Theorem \ref{compthm} with $\theta=\bar\theta\chi_{\Omega}$ for some $\bar\theta\in(\frac\pi 3,\frac 2 3 \pi]$\,. Then 
\begin{equation}\label{Gammaliminf}
\liminf_{\ep\to 0}\ep(\E_\ep(\mu_\ep)+{3}\,\mu_\ep(\R^2))\ge \Per_{\ffi_{\bar\theta}}(\Omega)\,.
\end{equation}
\item [(ii)] ($\Gamma$-limsup inequality) For every set $\Omega\subset\R^2$ of finite perimeter and for every $\bar\theta\in(\frac\pi 3,\frac 2 3 \pi]$\,, there exists a sequence $\{\mu_\ep\}\subset\M$
satisfying (i) and (ii) of Theorem \ref{compthm} with $\theta=\bar\theta\chi_{\Omega}$ 
 such that
\begin{equation}\label{Gammalimsup}
\limsup_{\ep\to 0}\ep(\E_\ep(\mu_\ep)+{3}\,\mu_\ep(\R^2))\le\Per_{\ffi_{\bar\theta}}(\Omega)\,.
\end{equation}
\end{itemize}
\end{theorem}
\begin{proof}
{\it Proof of (i).} We can assume without loss of generality that there exists $C<\infty$ such that
\begin{equation}\label{energybound}
\sup_{\ep>0} \ep(\E_\ep(\mu_\ep)+{3}\,\mu_\ep(\R^2))\le C\,.
\end{equation}

{\it Notation 1.}
For every $\ep>0$ set 
\begin{equation}\label{tutteetri}
G_\ep:=\bigcup_{f\in F_\ep(\mu_\ep)} f\quad\textrm{and}\quad\Omega_\ep:=\bigcup_{f\in F_\ep^{\Delta}(\mu_\ep)}f\,.
\end{equation}

Let $\delta\in (0,\frac 1 4)$\,; we classify the faces in $F_\ep^{\neq\Delta}(\mu_\ep)$ into two subclasses $\Pic$ and $\Gra$ defined by
\begin{equation}\label{perpicgra}
\Pic:=\left\{f\in F_\ep^{\neq\Delta}(\mu_\ep)\,:\,\Per(f)< \frac{\ep}{\delta}\right\}
\textrm{ and }
\Gra:=\left\{f\in F_\ep^{\neq\Delta}(\mu_\ep)\,:\,\Per(f)\ge \frac{\ep}{\delta}\right\}\,.
\end{equation}
Set moreover
\begin{align}\label{Oepdelta}
\Upic:=\bigcup_{f\in \Pic}f\,,\qquad \Ugra:=\bigcup_{f\in \Gra}f\,,\quad\textrm{and}\quad O_{\ep,\delta}:=\Omega_\ep\cup\Upic.
\end{align}
By construction $G_\ep=\Omega_\ep\cup \Upic\cup \Ugra=O_{\ep,\delta}\cup \Ugra$\,, where the unions are all disjoint, so that
\begin{equation}\label{lux}
\Per(O_{\ep,\delta})\le \Per(G_\ep)+\sum_{f\in \Gra} \Per(f)\,.
\end{equation}
{\it Claim 1:} $\ep(\E_\ep(\mu_\ep)+3\mu_\ep(\R^2))\ge \Per(O_{\ep,\delta})+\sum_{f\in \Pic} (\Per(f)-3\ep)+r(\delta)$\,, where $r(\delta)\to 0$ as $\delta\to 0$\,.

Indeed, by \eqref{energybound} and \eqref{energy2}, we have
\begin{equation}\label{faccegrandi}
C\ge \sum_{f\in \Gra}(\Per(f)-3\ep)\ge \ep\Big(\frac {1}{\delta}-3\Big) \sharp\Gra\,.
\end{equation}

Therefore, by \eqref{energy2}, \eqref{lux} and \eqref{faccegrandi}, it follows that
\begin{equation}\label{finefaccegrandi}
\begin{aligned}
\ep(\E_\ep(\mu_\ep)+3\mu_\ep(\R^2))&\ge
\Per(G_\ep)+\sum_{f\in F^{\neq\Delta}_{\ep}(\mu_\ep)} (\Per(f)-3\ep)\\
&\ge \Per(O_{\ep,\delta})+\sum_{f\in \Pic}(\Per(f)-3\ep)-3\ep \sharp \Gra\\
&\ge  \Per(O_{\ep,\delta})+\sum_{f\in \Pic}(\Per(f)-3\ep)-\frac{3C\delta}{1-3\delta}\,,
\end{aligned}
\end{equation}
which proves the claim with $r(\delta):=-\frac{3C\delta}{1-3\delta}$\,.
\vskip5pt
{\it Notation 2.}
Let $\eta\in (0,\frac{\pi}{6})$\,; we set
\begin{equation*}
\giu:=\bigcup_{\newatop{f\in F_{\ep}^{\Delta}(\mu_\ep)}{|\theta(f)-\bar\theta|<\eta}}f\,\qquad\textrm{and}\qquad
\sba:=\bigcup_{\newatop{f\in F_{\ep}^{\Delta}(\mu_\ep)}{|\theta(f)-\bar\theta|\ge\eta}}f\,.
\end{equation*}
Let
\begin{align*}
I_{\ep,\delta}^\eta:=\{f\in\Pic\,:\,\H^1(\partial f\cap\partial\sba)\ge\ep\}\,,\qquad
\tsba:=\sba\cup\bigcup_{f\in I_{\ep,\delta}^\eta}f\,.
\end{align*}
For every connected component $\gamma$ of (the closure of) $\tsba$ we set 
\begin{equation}\label{grandiepiccoli}
\begin{aligned}
\BorGra(\gamma)&:=\{\{x,y\}\in\Ed_\ep(\mu_\ep)\,:\,[x,y]\subset\partial\gamma\cap\partial O_{\ep,\delta}\}\,,\\
\NoBorGra(\gamma)&:=\{\{x,y\}\in\Ed_\ep(\mu_\ep)\,:\,[x,y]\subset\partial\gamma\setminus\partial O_{\ep,\delta}\}\,.
\end{aligned}
\end{equation}
We call $\zeta_1,\ldots,\zeta_{K_{\ep,\delta}^{\eta}}$ the connected components of $\tsba$ such that 
\begin{equation}\label{casouno}
\sharp\BorGra(\zeta_k)\ge\sharp \NoBorGra(\zeta_k)\,,
\end{equation}
and we set
$$
 \hat O_{\ep,\delta}^\eta:= O_{\ep,\delta}\setminus\bigcup_{k=1}^{K_{\ep,\delta}^{\eta}}\zeta_k\,.
$$
Finally, $\xi_1,\ldots,\xi_{J_{\ep,\delta}^{\eta}}$ denote the connected components of $\tsba$ such that
\begin{equation}\label{distinguo}
1\le\sharp\BorGra(\xi_j)<\sharp \NoBorGra(\xi_j)\,.
\end{equation}
Note that, by construction, the sets $\BorGra(\xi_j)$ and $\NoBorGra(\xi_j)$ do not change if in \eqref{grandiepiccoli}
 we replace $O_{\ep,\delta}$ by  $\hat O_{\ep,\delta}^{\eta}$\,.
 \begin{figure}[h]
{\def\svgwidth{200pt}
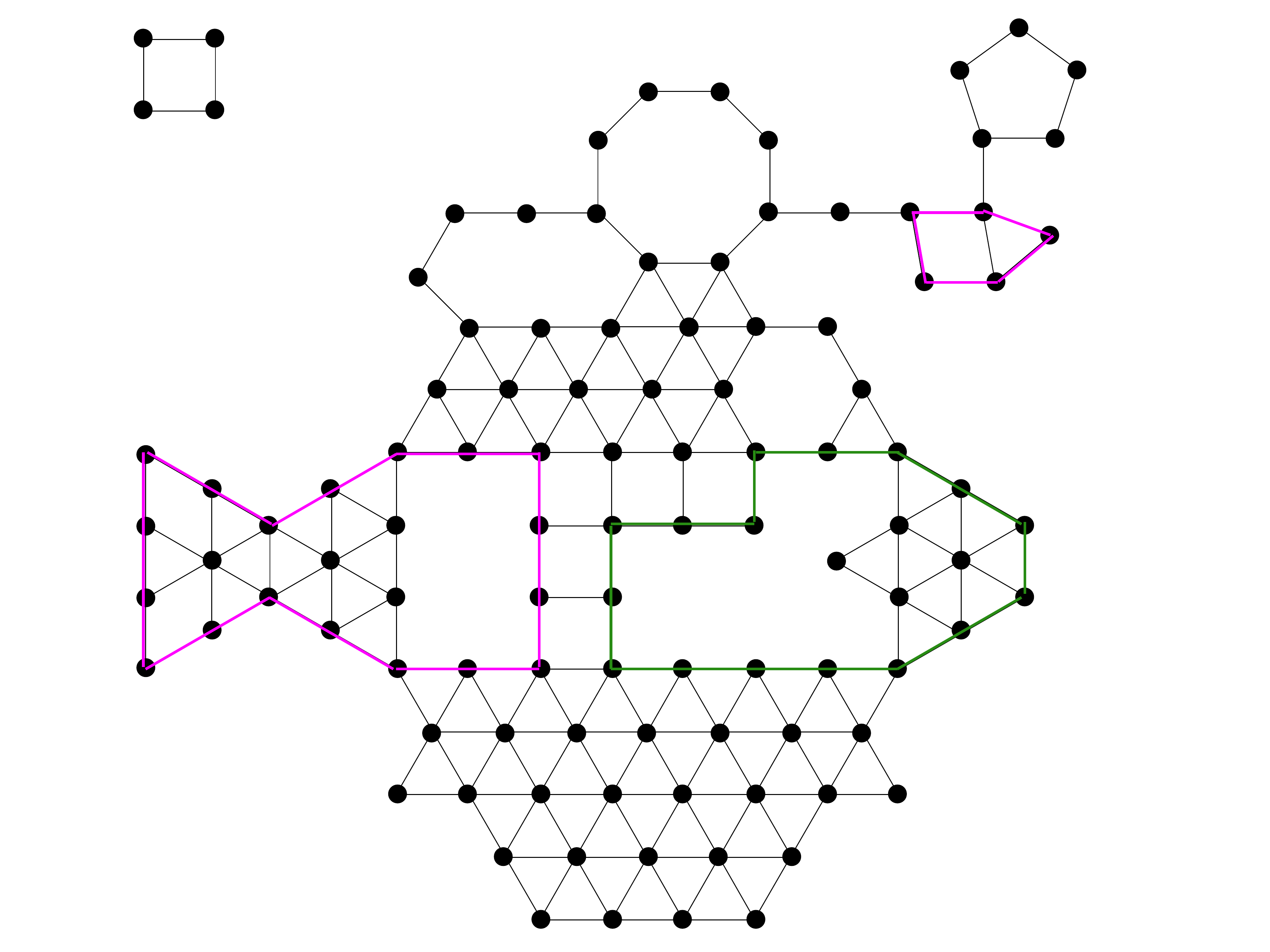}
\caption{The connected components $\zeta_k$ are contoured in pink\,, whereas $\xi_j$ is marked in green.}	\label{liminf fig}
\end{figure}

{\it Claim 2:} $\Per(O_{\ep,\delta})\ge \Per(\hat O_{\ep,\delta}^\eta)$\,.

Indeed, in view of \eqref{casouno},
\begin{align}\label{zeresima}
\Per( O_{\ep,\delta})=\Per( \hat O^{\eta}_{\ep,\delta})+\sum_{k=1}^{K_{\ep,\delta}^{\eta}}(\Per(\zeta_k)-2\ep\sharp\NoBorGra(\zeta_k))\ge \Per( \hat O^{\eta}_{\ep,\delta})\,.
\end{align}
\vskip5pt
{\it Notation 3.} Set
\begin{align*}
\BorBu&:=\{\{x,y\}\in \Ed_\ep(\mu_\ep)\,:\,[x,y]\subset\partial\giu\cap\partial \hat O^{\eta}_{\ep,\delta}\}\,,\\
\BorCa&:=\bigcup_{j=1}^{J_{\ep,\delta}^\eta}\BorGra(\xi_j)\,,\\
\BorMi&:=\{\{x,y\}\in \Ed_\ep(\mu_\ep)\,:\,[x,y]\subset\partial \hat O^{\eta}_{\ep,\delta}\}\setminus(\BorBu\cup\BorCa)\,.
\end{align*}
By construction the set of bonds $\{x,y\}$ with $[x,y]\subset\partial\hat O_{\ep,\delta}^\eta$ is given by the pairwise disjoint union of the bonds in the sets $\BorBu$\,, $\BorCa$\,, and $\BorMi$ so that 
\begin{align}\label{fondam}
\Per(\hat O_{\ep,\delta}):=\ep\sharp\BorBu+\ep\sharp\BorCa+\ep\sharp\BorMi\,.
\end{align}

{\it Claim 3:} $\Per(\hat O_{\ep,\delta}^\eta)+\sum_{f\in\Pic}(\Per(f)-3\ep)\ge (1+r(\eta))\Per_{\ffi_{\bar\theta}}(\hat O_{\ep,\delta}^\eta)$\,, where $r(\eta)\to 0$ as $\eta\to 0$\,.

We preliminarily notice that for every bond $\{x,y\}\in\BorBu$\,,
\begin{equation}\label{prima}
\ffi_{\bar\theta}([x,y])\le \frac{1}{1+r(\eta)}\,,
\end{equation}
for some $r(\eta)\to 0$ as $\eta\to 0$\,. 



Now
we prove that
\begin{equation}\label{terza}
\sum_{j=1}^{J_{\ep,\delta}^{\eta}}\sum_{\newatop{f\in \Pic}{f\subset\xi_j}}(\Per(f)-3\ep)\ge \sum_{j=1}^{J_{\ep,\delta}^{\eta}}\frac{\ep}{4}\sharp\BorGra(\xi_j)=\frac{\ep}{4}\sharp\BorCa\,.
\end{equation}
Indeed, let  $f\in\Pic$ with $f\subset\xi_j$ for some $j=1,\ldots, J_{\ep,\delta}^{\eta}$\,. Then
\begin{equation*}
\Per(f)-3\ep-\frac{\ep}{4}\sharp\{\{x,y\}\in\NoBorGra(\xi_j)\,:\,[x,y]\in\partial f\}\ge \frac{3}{4}\Per(f)-3\ep\ge 0\,.
\end{equation*}
It follows that
\begin{align*}
\Per(f)-3\ep\ge \frac{\ep}{4}\sharp\{\{x,y\}\in\NoBorGra(\xi_j)\,:\,[x,y]\in\partial f\}\,,
\end{align*}
which summing over $f\in \Pic$ with $f\subset\xi_j$\,, implies that
\begin{align}\label{quasiterza}
\sum_{\newatop{f\in \Pic}{f\subset\xi_j}}\Per(f)-3\ep\ge \frac{\ep}{4}\sum_{\newatop{f\in \Pic}{f\subset\xi_j}}\sharp\{\{x,y\}\in\NoBorGra(\xi_j)\,:\,[x,y]\in\partial f\}=\frac{\ep}{4}\sharp\NoBorGra(\xi_j)\,,
\end{align}
where the equality is a consequence of the fact that
if $\{x,y\}\in\NoBorGra(\xi_j)$ and $f\subset\Omega_{\ep}^{\eta+}$\,, then $\{x,y\}$ is not a bond of $f$\,.
By \eqref{quasiterza} and \eqref{distinguo}, and finally summing over $j$\,, we deduce \eqref{terza}.

Now, we consider  bonds $\{x,y\}\in\BorMi$\,, and we notice  that $[x,y]\in\partial f$ for some (unique)
$f\subset \hat O_{\ep,\delta}^\eta\setminus\bigcup_{j=1}^{J_{\ep,\delta}^\eta}\xi_j$ with $f\in\Pic$\,.
We prove that 
\begin{equation}\label{seconda} 
\sum_{\newatop{f\in\Pic}{f\subset\hat O_{\ep,\delta}^\eta\setminus\bigcup_{j=1}^{J_{\ep,\delta}^\eta}\xi_j}}(\Per(f)-3\ep)\ge\frac{\ep}{4}\sharp\BorMi\,.
\end{equation}
Indeed, let $f\in\Pic$ with $f\subset \hat O_{\ep,\delta}^\eta\setminus\bigcup_{j=1}^{J_{\ep,\delta}^\eta}\xi_j$ and let 
$\{x_1,y_1\}\,,\,\ldots,\{x_L,y_L\}\in \BorMi$ be such that $\cup_{l=1}^{L}[x_l,y_l]=\partial f\cap\partial \hat O^{\eta}_{\ep,\delta}$\,. Then, since $\Per(f)\ge 4\ep$\,, we have
\begin{equation}\label{seconda0}
\Per(f)-3\ep-\frac{\ep}{4}L=\frac{3}{4}\Per(f)+\frac{1}{4}(\Per(f)-\ep L)-3\ep\ge 0\,,
\end{equation}
which summing over $f$ implies \eqref{seconda}.

By \eqref{fondam}, \eqref{prima}, \eqref{seconda}, and \eqref{terza}, we can thus conclude
\begin{align*}
\Per(\hat O^{\eta}_{\ep,\delta})+\sum_{f\in\Pic}(\Per(f)-3\ep)&\ge\ep\sharp\BorBu+\ep\sharp\BorCa+\sum_{j=1}^{J_{\ep,\delta}^{\eta}}\sum_{\newatop{f\in \Pic}{f\subset\xi_j}}(\Per(f)-3\ep)\\
&+\ep\sharp\BorMi+
\sum_{\newatop{f\in\Pic}{f\subset\hat O_{\ep,\delta}^\eta\setminus\bigcup_{j=1}^{J_{\ep,\delta}^\eta}\xi_j}}(\Per(f)-3\ep)\\
&\ge (1+r(\eta))\sum_{\{x,y\}\in \BorBu}\ffi_{\bar\theta}([x,y])+\frac{5}{4}\ep\sharp \BorCa+\frac{5}{4}\ep\sharp\BorMi\\
&\ge  (1+r(\eta))\sum_{\{x,y\}\in \BorBu}\ffi_{\bar\theta}([x,y])+\sum_{\{x,y\}\in \BorCa\cup \BorMi}\ffi_{\bar\theta}([x,y])\\
&\ge (1+r(\eta))\Per_{\ffi_{\bar\theta}}(\hat O_{\ep,\delta}^{\eta})\,,
\end{align*}
where the third inequality is a consequence of the fact that $\frac{5}{4}\ge \frac{2}{\sqrt{3}}=\max_{v\in \mathbb S^1}\ffi_{\bar\theta}(v)$\,, being $\mathbb S^1$ the unitary sphere in $\R^2$\,.

\vskip5pt
{\it Claim 4:} $\displaystyle \lim_{\ep\to 0}\|\chi_{\hat O_{\ep,\delta}^\tau}-\chi_{\Omega}\|_{L^1(\R^2)}=0$\,.
By \eqref{energybound} and \eqref{energy2}, we get
\begin{equation}\label{numpic}
C\ge \sum_{f\in\Pic}(\Per(f)-3\ep)\ge\ep\sharp\Pic\,.
\end{equation}
Moreover, by the very definition of $\Pic$ in \eqref{perpicgra} and by the isoperimetric inequality, we have
\begin{equation}\label{areapic}
|f|\le \frac{\ep^2}{4\pi \delta^2}\qquad\textrm{for every }f\in\Pic\,.
\end{equation}
Furthermore, by assumption
$$
\lim_{\ep\to 0}|\Omega_{\ep}^{\eta+}|=0\,,
$$
which, combined together with \eqref{numpic} and \eqref{areapic}, yields
$$
|\Omega_\ep\Delta \hat O_{\ep,\delta}^\eta|\le \sum_{f\in\Pic}|f|+|\Omega_{\ep}^{\eta+}|\le \frac{C\,\ep^2}{4\pi\delta^2}+|\Omega_{\ep}^{\eta+}|\to 0\,,
$$
whence the claim immediately follows.
\vskip5pt
{\it Conclusion: \eqref{Gammaliminf} holds true.}
By Claims 1-3, we have
$$
\ep(\E_\ep(\mu_\ep)+3\mu_\ep(\R^2))\ge (1+r(\eta))\Per_{\ffi_{\bar\theta}}(\hat O_{\ep,\delta}^\eta)+r(\delta)\,,
$$
which, by Claim 4, in view of the lower semicontinuity of the anisotropic perimeter with respect to the strong convergence in $L^1(\R^2)$\,, implies 
$$
\liminf_{\ep\to 0}\ep(\E_\ep(\mu_\ep)+3\mu_\ep(\R^2))\ge (1+r(\eta))\Per_{\ffi_{\bar\theta}}(\Omega)+r(\delta)\,.
$$
Then
\eqref{Gammaliminf} by sending $\eta\to 0$ and $\delta\to 0$\,.
\vskip10pt
{\it Proof of (ii).} 
The $\Gamma$-limsup inequality can be easily obtained as a  consequence of \eqref{Gammalimsupgen} and \eqref{energy2}. For the reader's convenience, we briefly sketch the proof.  
By standard density arguments in $\Gamma$-convergence we can assume that $\Om$ has a finite number $M$ of connected components with polyhedral boundary. Let $X^{\bar \theta}_\e$ be the periodic lattice generated by $\ep e^{i (\bar\theta - \frac{\pi}{2})}$ and by  $\ep e^{i (\bar\theta - \frac{\pi}{6})}$\,. We denote by $F_\ep(X_\ep^{\bar\theta})$ the set of equilateral triangles with vertices in $X^{\bar \theta}_\e$ and side-length equal to $\ep$\,. Set 
$$
X_\e:= \{ x\in f: f\in F_\e(X^{\bar \theta}_\e), \, f \subset \Om \}, \qquad \mu_\e:= \sum_{x\in X_\e} \delta_x, \qquad  \Om_\e:= \bigcup_{f \in F_\e(X^{\bar \theta}_\e): \, f\subset \Om} f\,.
$$ 
Since $F_\ep(X_\ep)=F_\ep^{\Delta}(X_\ep)$\,, by \eqref{energy2}, we immediately have 
$$
\ep(\E_\ep(\mu_\ep)+{3}\,\mu_\ep(\R^2))= \Per (\Om_\e)+3\ep M.  
$$
Moreover one can trivially check that $\Per (\Om_\e) \to \Per_{\ffi_{\bar\theta}}(\Om)$ as $\e\to 0$\,, thus concluding the proof of (ii).
\end{proof}
Finally, in this last part of the section we briefly consider the case of additional confining forcing terms. Let
$$
\F^g_\e(\mu):= \E_\e(\mu) + 3 \mu(\R^2)+ \frac{\sqrt{3} }{2 \e} \int_{\R^2} g  \ud\mu\,,
$$
where $g\in \C^0(\R^2)$ with $g(x)\to +\infty$ as $|x|\to +\infty$\,.

\begin{corollary}
Let $\{\mu_\ep\}\subset\M$ be such that $\F^g_\ep(\mu_\ep) \le \frac C \ep$\,. Then,  up to a subsequence, 
\begin{itemize}
\item[(1)](Compactness for $\mu_\e$) $\ep^2\frac{\sqrt 3}{2}\mu_\ep\weakstar\chi_\Omega \ud x$ for some set $\Omega\subset\R^2$ with $\chi_{\Omega}\in SBV(\R^2)$\,. Moreover, $\ep^2\frac{\sqrt 3}{2} \mu_\ep (\R^2)\to |\Omega|$\,. 
\item[(2)](Compactness for $\theta_\e$) $\theta_\ep(\mu_\ep)\weakly \theta$ in $SBV(\R^2)$\,, for some $\theta=\sum_{j\in J}\theta_j\chi_{\omega_j}$\,, where $J\subseteq\N$\,,
$\{\omega_j\}_j$ is a Caccioppoli partition of $\Omega$\,, and $\{\theta_j\}_j\subset (\frac\pi 3,\frac 2 3\pi]$\,.
	\item [(3)] ($\Gamma$-liminf inequality) If $\theta=\bar\theta\chi_{\Omega}$ for some $\bar\theta\in(\frac\pi 3,\frac 2 3 \pi]$\,, then 
	\begin{equation}\label{Gammaliminf3}
	\liminf_{\ep\to 0}\ep \F^g_\ep(\mu_\ep) \ge \Per_{\ffi_{\bar\theta}}(\Omega) + \int_\Om g \ud x\,.
	\end{equation}
	\item [(4)] ($\Gamma$-limsup inequality) For every set $\Omega\subset\R^2$ of finite perimeter and for every $\bar\theta\in(\frac\pi 3,\frac 2 3 \pi]$\,, there exists a sequence $\{\mu_\ep\}\subset\M$
	satisfying (1) and (2)  with $\theta=\bar\theta\chi_{\Omega}$ 
	such that
\begin{equation}\label{Gammalimsup3}
	\limsup_{\ep\to 0}\ep \F^g_\ep(\mu_\ep)\le\Per_{\ffi_{\bar\theta}}(\Omega) + \int_\Om g \ud x\,.
	\end{equation}
	
\end{itemize}
\end{corollary}
\begin{proof}
We briefly sketch the proof, the details are left to the reader. Items (1) and (2) are an easy consequence of Theorem \ref{compthm} and of the fact that, in view of the coercivity assumption $g(x)\to +\infty$ as $|x|\to +\infty$\,, there is no loss of mass at infinity. 
Items (3) and (4) are consequences of Theorem \ref{gammaconvHR}, once noticed that the functionals $\F_\e^g(\mu)$ are nothing but the functionals $\E_\e(\mu) + 3 \mu(\R^2)$  plus the continuous perturbation  $\frac{\sqrt{3}}{2 \e} \int_{\R^2} g\ud x$\,. 
\end{proof}


\section{Asymptotic behaviour of energy minimizers}
In this section, we present some variational problems for which the asymptotic behaviour of minimizers can be easily studied  using Theorems \ref{compthm} and \ref{gammaconvHR}.

\subsection{Energy bounds for polycrystals}
\begin{proposition}\label{pcno}
The following lower and upper bounds hold true.
\begin{itemize}
\item [(i)] (Lower bound) 
For all $\{\mu_\ep\}\subset\M$ satisfying (i) and (ii) of Theorem \ref{compthm},
we have
\begin{equation}\label{Gammaliminfp}
\liminf_{\ep\to 0}\ep(\E_\ep(\mu_\ep)+{3}\,\mu_\ep(\R^2))\ge \H^1(\partial^*\Om)+ \frac 1 2 \H^1( \cup_j \partial^* \omega_j \setminus \partial^*\Om)\,.
\end{equation}

\item [(ii)] (Upper bound) For every set $\Omega\subset\R^2$ of finite perimeter and for every $\theta\in SBV(\Om;(\frac\pi 3,\frac 2 3 \pi])$  there exists a sequence $\{\mu_\ep\}\subset\M$
satisfying (i) and (ii) of Theorem \ref{compthm} 
 such that
\begin{equation}\label{Gammalimsupp}
\limsup_{\ep\to 0}\ep(\E_\ep(\mu_\ep)+{3}\,\mu_\ep(\R^2))\le \sum_j \Per_{\ffi_{\theta_j}}(\omega_j)\,.
\end{equation}
\end{itemize}
\end{proposition}
\begin{proof}
We start by proving (i). 
Let $\delta\in (0,\frac{1}{4})$ and let $O_{\ep,\delta}$ and $\Omega_\ep$ be defined as in \eqref{Oepdelta} and \eqref{tutteetri}. By Claim 1 in the proof of Theorem \ref{gammaconvHR} and by using that $\Per(f)\ge 4\ep$ for every $f\in F_\ep^{\neq\Delta}(\mu_\ep)$\,, we have
$$
\ep(\E_\ep(\mu_\ep)+{3}\,\mu_\ep(\R^2))\ge\Per (O_{\ep,\delta})+\frac 1 4 \H^1(\partial\Omega_\ep\setminus\partial O_{\ep,\delta})+r(\delta)\,,
$$
where $r(\delta)\to 0$ as $\delta\to 0$\,.
By arguing as in Claim 4 in the proof of Theorem \ref{gammaconvHR} one can easily show that
\begin{equation}\label{convl1oep}
\|\chi_{O_{\ep,\delta}}-\chi_{\Omega}\|_{L^1}\to 0\qquad\textrm{as }\ep\to 0\,.
\end{equation}
Let $\bar\mu_\ep$ be the Radon measure defined by 
$$
\bar\mu_\ep(A):=\Per (O_{\ep,\delta}, A)+\frac 1 4 \H^1((\partial\Omega_\ep\setminus\partial O_{\ep,\delta})\cap A)\qquad\textrm{for every open set }A\subset\R^2\,.
$$
By \eqref{convl1oep} and by the lower semicontinuity of the relative perimeter, we have
\begin{equation}\label{borext}
\liminf_{\ep\to 0}\bar\mu_\ep(B_r(x))\ge \Per(\partial^*\Omega,B_r(x))\qquad \textrm{for every }x\in\partial^*\Omega \textrm{ and for every }r>0\,.
\end{equation}
Let now $x\in (\partial^*\omega_j\cap \partial^*\omega_k)\setminus\partial^*\Omega$ for some $j,k\in\N$ with $j\neq k$ and let $r>0$\,.
By \eqref{Gammaliminfgen} we have
\begin{equation}\label{borint}
\liminf_{\ep\to 0}\bar\mu_\ep(B_r(x))\ge \frac 1 4 \Per(\partial^*\omega_j,B_r(x))+ \frac 1 4 \Per(\partial^*\omega_k,B_r(x)) \,.
\end{equation}
By \eqref{borext} and \eqref{borint} together with standard blow up arguments  \eqref{Gammaliminfp} follows.

{Finally, we briefly sketch the proof of (ii). Again by standard density arguments in $\Gamma$-convergence, we can assume that the $\omega_j$'s are in a finite number $M$\,, have pairwise disjoint closures and have polyhedral boundaries. Then, denoting  by $\mu_\ep^j$ the measure constructed in (ii) of Theorem \ref{gammaconvHR} for $\Omega=\omega_j$\,, it is easy to check that $\mu_\ep:=\sum_{j}\mu_\ep^j$ satisfies \eqref{Gammalimsupp}.}
\end{proof}

\subsection{Single crystals versus polycrystals}\label{svp}

\begin{corollary}\label{corosc}
Let $\bar \theta \in (\frac\pi 3,\frac 2 3 \pi]$\,, and let  $\Om$ be a subset of $\R^2$ with finite perimeter such  that $\nu(x)\in \{ v_{k,\bar \theta} \}_{k=1,2,3}$ for $\H^1$-a.e. $x\in \partial^*\Om$\,, 
where $v_{k,\bar \theta} = e^{i(\bar \theta-\frac \pi 2)} v_k$\,, with $v_k$ defined in \eqref{gener}. 

Let $\e_n\to 0$ and $\{\mu_{\e_n}\} \subset \A$ be such that
\begin{equation}\label{c1i}
\inf_{\e^2 \frac{\sqrt{3}}{2} \lambda_\e \weakstar \chi_\Om} \liminf_{\e\to 0} \e( \E_\e (\lambda_{\e}) + 3  \lambda_{\e}(\R^2))=
\lim_{n\to \infty} \e_n( \E_{\e_n} (\mu_{\e_n}) + 3  \mu_{\e_n}(\R^2)) \, .
\end{equation}

Then, up to a subsequence,  $\theta_{\ep_n}(\mu_{\ep_n}) \weakly \bar \theta \chi_\Om $ in $SBV_\loc(\R^2)$\,, where $\theta_{\ep_n}(\mu_{\ep_n})$ is defined according with \eqref{deftheta}. 
\end{corollary}
\begin{figure}[h!]
	\centering
	\def\svgwidth{0.5\textwidth}
	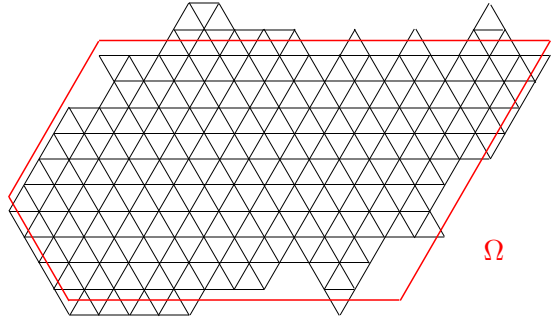
	\caption{Single crystal.}	\label{sincry}
\end{figure}
\begin{proof}
By \eqref{Gammalimsup} it easily follows that $\e_n( \E_{\e_n} (\mu_{\e_n}) + 3  \mu_{\e_n}(\R^2)) \le C$\,. By Theorem \ref{compthm}, we have that, 
up to a subsequence,  $\theta_{\ep_n}(\mu_{\ep_n}) \weakly  \theta  $ in $SBV_\loc(\R^2)$ for some $\theta\in SBV(\Om)$ with $\theta=\sum_{j\in J}\theta_j\chi_{\omega_j}$\,, where $J\subseteq\N$\,,
$\{\omega_j\}_j$ is a Caccioppoli partition of $\Omega$\,, and $\{\theta_j\}_j\subset (\frac\pi 3,\frac 2 3\pi]$\,.
 By \eqref{Gammaliminfp},\eqref{c1i} and  \eqref{Gammalimsup} we have
\begin{multline*}
\Per_{\ffi_{\bar \theta}}(\Om) + \frac 1 2 \H^1( \cup_j \partial^* \omega_j \setminus \partial^*\Om) =
\H^1(\partial^*\Om)+ \frac 1 2 \H^1( \cup_j \partial^* \omega_j \setminus \partial^*\Om)
 \\
\le \liminf_{n\to\infty}\ep_n(\E_{\ep_n}(\mu_{\ep_n})+{3}\,\mu_{\ep_n}(\R^2))
=\inf_{\e^2 \frac{\sqrt{3}}{2} \lambda_\e \weakstar \chi_\Om} \liminf_{\e\to 0} \e( \E_\e (\lambda_{\e}) + 3  \lambda_{\e}(\R^2)) 
\le \Per_{\ffi_{\bar \theta}}(\Om)\,.
\end{multline*}
We deduce that $ \frac 1 2 \H^1(\{ \cup_j \partial^* \omega_j \setminus \partial^*\Om\}) =0$\,, and hence $\theta=\hat \theta\chi_{\Omega}$ for some $\hat \theta \in (\frac{\pi}{3},\frac{2}{3}\pi]$\,. 
By  \eqref{Gammaliminf} and \eqref{Gammalimsup} we deduce  that
\begin{multline*}
\Per_{\ffi_{\bar \theta}}(\Om) \le \Per_{\ffi_{\hat \theta}}(\Om) \le \liminf_{n\to\infty}\ep_n(\E_{\ep_n}(\mu_{\ep_n})+{3}\,\mu_{\ep_n}(\R^2))
\\
=\inf_{\e^2 \frac{\sqrt{3}}{2} \lambda_\e \weakstar \chi_\Om} \liminf_{\e\to 0} \e( \E_\e (\lambda_{\e}) + 3  \lambda_{\e}(\R^2)) 
{\le \Per_{\ffi_{\bar \theta}}(\Om)}\,.
\end{multline*}
We conclude that $\Per_{\ffi_{\bar \theta}}(\Om) = \Per_{\ffi_{\hat \theta}}(\Om)$\,, which implies $\hat \theta = \bar \theta$\,.
\end{proof}
\begin{remark}
Using the minimality property of the measures $\mu_{\e_n}$\,, one can prove that the compactness of the sequence $\{\mu_{\e_n}\}$ stated  in Corollary \ref{corosc} in fact holds true in $SBV(\R^2)$\,. 
\end{remark}
\begin{corollary}
Let $\vth_1\,,\,\vth_2\in (\frac\pi 3,\frac 2 3 \pi]$ with $\vth_1\neq \vth_2$ and, given $\tau\in\R^2$\,, set 
\begin{equation}\label{defm}
\Om_\tau:= W_{\vth_1} \cup (W_{\vth_2} + \tau)\,, \qquad m(\tau):= | W_{\vth_1} \cap (W_{\vth_2} + \tau)|\,.
\end{equation}
Then, there exists $\bar m = \bar m (\vth_1, \vth_2)$ such that, whenever $m(\tau) \le \bar m$ the following holds:

Let $\e_n\to 0$ and $\{\mu_{\e_n}\} \subset \A$ be such that
$$
\inf_{\e^2 \frac{\sqrt{3}}{2} \lambda_\e \weakstar \chi_{\Om_\tau}} \liminf_{\e\to 0} \e( \E_\e (\lambda_{\e}) + 3  \lambda_{\e}(\R^2))=
\lim_{n\to \infty} \e_n( \E_{\e_n} (\mu_{\e_n}) + 3  \mu_{\e_n}(\R^2)) \, .
$$ 

Then, up to a subsequence,  $\theta_{\ep_n}(\mu_{\ep_n}) \weakly \theta$ in $SBV_\loc(\R^2)$\,, for some $\theta=\sum_{j\in J}\theta_j\chi_{\omega_j}$ in $SBV(\R^2)$\,, where $J\subseteq\N$\,,
$\{\omega_j\}_j$ is a Caccioppoli partition of $\Omega_\tau$\,, and $\{\theta_j\}_j\subset (\frac\pi 3,\frac 2 3\pi]$\,. Moreover, the function $\theta$ is not constant, i.e., $\sharp J\ge 2$\,. 
\end{corollary}  
\begin{figure}[h!]
	\centering
	\def\svgwidth{0.5\textwidth}
	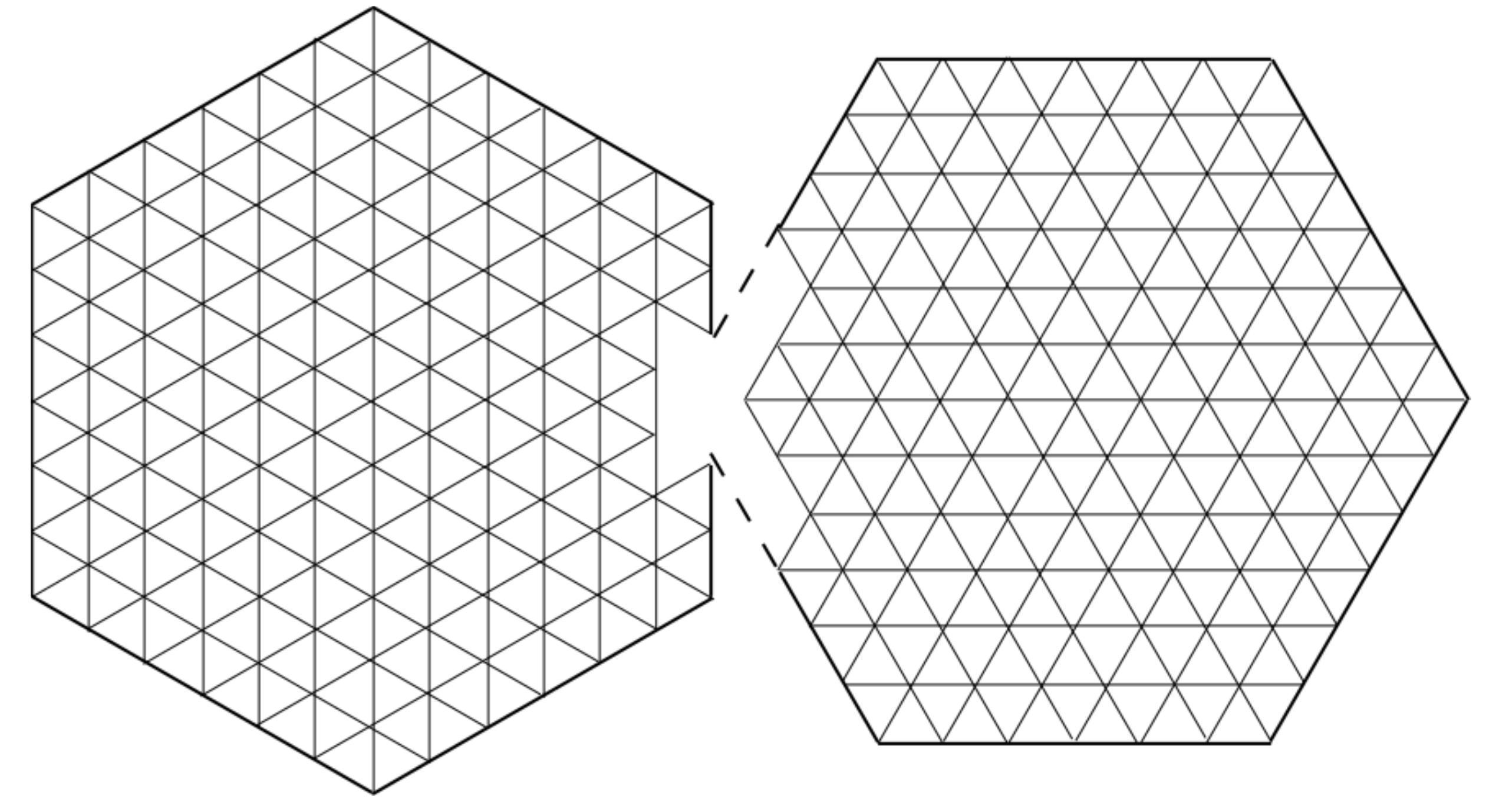
	\caption{ polycrystal.}	\label{policry}
\end{figure}
\begin{proof}
Notice that we can always write $\Om_\tau=\omega_1 \cup \omega_2$\,, with $\omega_1$ and $\omega_2$ open disjoint set such that $S:=\partial \omega_1 \cap \partial \omega_2$ is a segment and $\partial \omega_j \setminus S \subset W_{\vth_j}$ for $j=1,2$\,. Moreover, there exists a modulus of continuity $l(m)\to 0$ as $m\to 0$
such that  $\H^1(S)\le l(m)$ with $m=m(\tau)$ defined in \eqref{defm}.

By \eqref{Gammalimsupp} there exists a sequence $\{\bar \lambda_\e\}$ such that 
\begin{multline}\label{smallm1}
\lim_{n\to\infty} \e_n( \E_{\e_n} (\mu_{\e_n}) + 3  \mu_{\e_n}(\R^2))= \inf_{\e^2 \frac{\sqrt{3}}{2} \lambda_\e \weakstar \chi_{\Om_\tau}} \liminf_{\e\to 0} \e( \E_\e (\lambda_{\e}) + 3  \lambda_{\e}(\R^2))
\\
\le
\limsup_{\ep\to 0}\ep(\E_\ep(\bar\lambda_\ep)+{3}\,\bar\lambda_\ep(\R^2))\le \sum_{j=1}^2 \Per_{\ffi_{\vth_j}}(\omega_j)
\le \sum_{j=1}^2 \Per(W_{\vth_j}) + c\, l (m)\,.
\end{multline}
for some $c<\infty$\,.
In particular, by \eqref{compthm}, $\theta_{\ep_n} (\mu_{\ep_n}) \weakly \theta$ in $SBV_\loc(\R^2)$ for some 
$\theta\in SBV(\Om)$ with $\theta=\sum_{j\in J}\theta_j\chi_{\omega_j}$\,, where $J\subseteq\N$\,,
$\{\omega_j\}_j$ is a Caccioppoli partition of $\Omega$\,, and $\{\theta_j\}_j\subset (\frac\pi 3,\frac 2 3\pi]$\,. It remains to prove that, for $m(\tau)$ small enough, $\theta$ is not constant. If $\theta=\bar \theta\chi_{\Omega_\tau}$ for some $\bar\theta\in(\frac{\pi}{3},\frac{2}{3}\pi]$\,, then, by \eqref{Gammaliminf}, we have

\begin{multline}\label{smallm2}
\lim_{n\to \infty} \e_n( \E_{\e_n} (\mu_{\e_n}) + 3  \mu_{\e_n}(\R^2))\ge \Per_{\ffi_{\bar\theta}}(\Om_\tau)
\\
\ge 
\sum_{j=1}^2 (1+ p(\bar\theta - \theta_j)) \H^1(\partial\Omega_\tau \cap \partial W_{\theta_j}) \ge 
\sum_{j=1}^2 (1+ p(\bar\theta - \theta_j)) \Per(\partial W_{\theta_j}) -  r(m)\,,
\end{multline}
for some moduli of continuity $p\,, \, r:[0,+\infty)\to \R$ which are continuous, vanishing at zero and strictly positive elsewhere. 
Clearly   \eqref{smallm1} and \eqref{smallm2} are not compatible for $m$ smaller than some $\bar m$ depending only on $\vth_1$ and $\vth_2$\,.
\end{proof}
\appendix
\section{Optimal tessellations of the plane}
It is well known that the plane can be tessellated by regular polygons; more precisely, by equilateral triangles, squares and hexagons. 

Fix one of such regular polygons $p$ 
and assume that the edges of $p$ have length equal to one. Let $I=(\Theta_1,\Theta_2]$ be a given  interval, representing the family of  orientations of $p$ , and satisfying suitable  properties listed below. Set $\Theta_{\av}:= \frac12 (\Theta_1+\Theta_2)$\,,  and  
$p_\theta:= e^{i(\theta - \Theta_{\av})} p$ for all $\theta \in \bar I$\,.  The required properties of $I$ are that   $0\notin I$\,, $\H^1(\partial (p_{\Theta_1}+\tau) \cap \partial p_{\Theta_2})=1$ for some $\tau\in\R^2$\,, $\H^1(\partial (p_{\alpha_1}+\tau) \cap \partial p_{\alpha_2})=0$ for every $\alpha_1\,,\alpha_2\in I$ and for every $\tau\in\R^2$\,.

For instance, we can choose $I=(\frac{\pi}{3}, \frac{2}{3}\pi]$  if $p$ is the equilateral triangle or the regular hexagon and $I=(\frac{\pi}{4}, \frac{3}{4}\pi]$  if $p$ is the  square.

Let $\ffi$ anf $\ffi_\theta$ be defined as in \eqref{defffi} and  \eqref{defffian}, with the vectors  $v_i$ in \eqref{gener} replaced by the 
normals $\nu_k$ to $\partial p$\,, i.e., the unitary vectors orthogonal  to the edges of $p$\,. 

For every $\ep >0$ set
$$
F_\ep:= \{ \ep p_\theta + \tau \,:  \, \theta \in I,\, \tau \in \R^2\}\,.
$$
Notice that for all $f\in F_\ep$ there exists a unique $\theta=\theta(f)$ such that $f= \ep p_\theta$ up to a (still unique) translation.

\begin{lemma}\label{a1}
Let $\{\nu_k\}$ be the set of the normals to  $\partial p$\,.
There exists a modulus of continuity $r(\eta)$ with the following property. 
Let $\ffi^\eta: \mathbb S^1\to \R$ be defined by
$$
\ffi^\eta(v) :=
\begin{cases}
1 & \text{ if } \max_k | v \cdot \nu_k |  \ge 1- \eta\,; 
\\
\ffi(v) & \text{ otherwise.} 
\end{cases}
$$
Let 
$\{\Om_\e\}$ be a sequence of  sets of finite perimeter such that 
$\chi_{\Om_\ep} \to \chi_\Om$ in $SBV_\loc(\R^2)$ for some set $\Om$ of finite perimeter.
Then
$$
\liminf_{\ep\to 0} \int_{\partial^{*}\Om_\ep}\ffi^\eta (\nu)\ud \H^1
\ge (1+r(\eta))\Per_{\ffi}(\Omega)\,.
$$
\end{lemma}
\begin{proof}
There exists $c(\eta)>0$ with $c(\eta)\to 1$ as $\eta\to 0$ such that  $|\ffi^\eta(v) - \ffi (v)|\le c(\eta) $  for all $v\in \mathbb S^1$\,. Therefore the lemma is an easy consequence of the lower semicontinuity of the $\ffi$-perimeter $\Per_\ffi$\,.
\end{proof}
In what follows, for every $\ep>0$\,, we denote by $\Phi_\ep$ the set of family of faces $H_\ep\subset F_\ep$ whose interiors are pairwise disjoint.
Moreover, we set
\[
\O_\ep:=\Big\{\Om\subset\R^2\,:\, \Om=\bigcup_{f\in H_\ep} f,\  \text{for some } H_\ep\in \Phi_\ep
\Big\}\,.
\]

Now we prove a $\Gamma$-convergence result.

\begin{theorem}\label{teoapp}
The following $\Gamma$-convergence result holds true.
\begin{itemize}
\item[(a)] (Compactness) Let $H_\ep \in \Phi_\ep$ and set  
$$
\Omega_\ep := \bigcup_{f\in H_\ep} f\,, \qquad \theta_\ep:= \sum_{f\in H_\ep} \theta(f) \chi_f\,.
$$
Assume that $\Per (\Om_\ep)\le C$\,. Then, up to a subsequence, $\chi_{\Om_\ep} \weakly \chi_\Om$ in $SBV_\loc(\R^2)$ for some set $\Om$ of finite perimeter. Moreover,  
$\theta_\ep \weakly \theta$ in $SBV_\loc(\R^2)$\,, for some $\theta=\sum_{j\in J}\theta_j\chi_{\omega_j}$ in $SBV(\Omega)$\,, where $J\subseteq\N$\,,
 $\{\omega_j\}_j$ is a Caccioppoli partition of $\Omega$\,, and $\{\theta_j\}_j\subset I$\,.
\item [(b)] ($\Gamma$-liminf inequality) Let $H_\ep$\,,   $\Om_\ep$\,, $\theta_\ep$\,, $\Om$\,, and  $\theta$ be as in (a), and let $A$ be a an open set. Then
\begin{equation}\label{Gammaliminfgen}
\liminf_{\ep\to 0}\Per(\Omega_\ep,A)\ge 
\sum_{j\in J}\Per_{\ffi_{\theta_j}}(\omega_j,A)\,.
\end{equation}
\item [(c)] ($\Gamma$-limsup inequality) For every set $\Omega\subset\R^2$ of finite perimeter and for every $\theta\in SBV(\Om)$\,, there exists a sequence $\{H_\ep\}$ satisfying the claim in (a) 
 such that
\begin{equation}\label{Gammalimsupgen}
\limsup_{\ep\to 0}\Per(\Omega_\ep)\le
\sum_{j\in J}\Per_{\ffi_{\theta_j}}(\omega_j)\,.
\end{equation}
\end{itemize}
\end{theorem}
\begin{proof}
{\it Proof of (a).} By the very definition of $\theta(f)$\,, we have that $\|\theta_\ep\|_{L^\infty}\le \Theta_2$\,. Moreover, by the uniform bound on $\Per(\Omega_\ep)$\,, we obtain that $\H^1(S_{\theta_\ep})\le C$\,. It follows that $\|\theta_\e\|_{BV} \le C$ for some constant $C<\infty$ independent of $\e$\,. Then the claim follows from Corollary \ref{SBVcompcor}.

{\it Proof of (b).} For every $\vth\in I$ let $H_\e(\vth):= \{f\in H_\ep : \, \theta(f) = \vth\}$\,, and set $\Om_\e(\vth):= \sum_{f\in H_\e(\vth)} \chi_f$\,. Notice that if $\vth_1\neq \vth_2$\,, then 
$$
\Per(\Om_\e(\vth_1)\cup\Om_\e(\vth_2),A)=
\Per(\Om_\e(\vth_1),A)+ \Per(\Om_\e(\vth_2),A),
$$
for every open bounded set $A\subset\R^2$\,.
It follows that there exists an at most countable set of indices $J$ and a set $\{\vth_n\}_{n\in J}\subset I$ such that $\Om_\e(\vth_n) \neq \emptyset$ for every $n\in J$ and  
$$
\Per(\Om_\e,A) = \sum_{n\in J}
\Per(\Om_\e(\vth_n),A)\quad\textrm{for every open bounded set }A\subset\R^2\,.
$$
Let $M\in\N$ and consider $\vth_1\,\ldots, \vth_M\in J$\,. Let $\eta>0$ be such that 
$$
|\vth_i-\vth_j|>\eta \qquad \text{ for all } 1\le i<j\le M\,.
$$
Moreover, for every $1\le i\le M$ set $I^\eta_{\e,i}:=\{\vth\in I : |\vth-\vth_i|<\frac{\eta}{2}\}$\,, and 
$$
\Om_{\e,i}^{\eta}:=\bigcup_{\vth\in I^\eta_{\e,i}} \Om_\e(\vth)\,.
$$
Then,  $\chi_{\Om_\e^{i,\eta}}\to \chi_{\Om^\eta_i}$
in $L^1_{\loc}(\R^2)$\,, with 
$$
\Om^\eta_i:= \bigcup_{j\in J\,:\, |\vth_j -\vth_i|\le\frac{\eta}{2}} \omega_j\,. 
$$
Trivially, for every $i=1,\ldots,M$ we have that $\Om^\eta_i \to \omega_i$ in $L^1_{\loc}(\R^2)$ as $\eta\to 0$ { and  $|\theta_i-\vth_i|\le \frac{\eta}{2}$}\,.
By Lemma \ref{a1} we deduce that for every open bounded set $A$ it holds
$$
\liminf_{\ep\to 0}\Per(\Omega_\ep,A)\ge 
\sum_{i=1}^M \liminf_{\ep\to 0}\Per(\Omega_{\ep,i}^{\eta},A)\ge (1+r(\eta)) 
\sum_{i=1}^M \Per_{\ffi_{\theta_i}}(\Omega^\eta_i,A)\,.
$$
Letting first $\eta\to 0$ and then $M\to +\infty$   we deduce the $\Gamma$-liminf inequality (b).\\

{\it Proof of (c).} Since partitions with polyhedral boundary are dense (see \cite{BCG}), by standard density arguments in $\Gamma$-convergence we can assume that $\omega_i$ are polygons. In this case, the construction of a recovery sequence satisfying \eqref{Gammaliminfgen} follows by arguing as in the proof of \eqref{Gammalimsupp}.
\end{proof}
For $\eps>0$ we can define the following functional
\[
\Per_\ep(\Om)=\left\{\begin{array}{ll}
\Per(\Om)&\textrm{if }\Om\in\O_\eps\,,
\\
+\infty&\textrm{otherwise.}
\end{array}\right.
\]

We state the following corollary which is a direct consequence of Theorem \ref{teoapp}.
\begin{corollary}
The functionals $\Per_\ep$ $\Gamma$-converge, with respect to the convergence in $L^1_{\loc}(\R^2)$ of characteristic functions, to the functional   $\Per_0$ defined by
\begin{equation*}
\begin{aligned}
\Per_0(\Om) &:= \min \Big \{ \sum_{j\in J}\Per_{\ffi_{\theta_j}}(\omega_j)\,: \,   \sum_{j\in J} \theta_j \chi_{\omega_j} \in SBV(\Om)\,, \\ 
&\phantom{\sum_{j\in J}\Per_{\ffi_{\theta_j}}(\omega_j)}\{\omega_j\}_{j} \text{ is a Caccioppoli partition of } \Om
\,,\,\{\theta_j\}_j\subset I\Big \}
\,,
\end{aligned}
\end{equation*}
for all sets $\Om$ with finite perimeter.
\end{corollary}

Clearly we have the following inequality, valid for all sets $\Om$ with finite perimeter:
\begin{equation}\label{sop}
\Per_0(\Om) \le \min_{\theta\in I}\Per_{\ffi_{\theta}}(\Omega)\,.
\end{equation}

The considerations in Subsection \ref{svp} can be easily extended to the functionals $\Per_\e$ and $\Per_0$\,. In particular, there exist sets $\Om$ such that the inequality in \eqref{sop} is strict and, on the other hand, 
\[
\min_{\Om:\,|\Om|=m} \Per_0(\Om) = \sqrt{m}\,\Per(W_\ffi) \, ,
\]
where $W_\ffi$ is the Wulff shape which  solves the  isoperimetric problem corresponding to the anisotropy $\ffi$\,, among sets with unit area. 
Moreover, letting $\Omega_\eps$ be minimizers of 
\[
\min_{\Om:\,|\Om|=m} \Per_\eps(\Om) \,,
\]
it follows that, up to rotations and translations, $\Omega_\eps$ converge
to $ \sqrt{m}\,W_\ffi$ in the $L^1$-topology.



\end{document}